\documentclass[11pt]{article}

\usepackage{amsmath,amsthm,amssymb,graphicx,xcolor}
\usepackage{fullpage}
\usepackage{mathtools}
\usepackage{color}
\usepackage{enumitem}
\usepackage{comment}
\usepackage{setspace}
\usepackage[colorlinks,citecolor=blue,urlcolor=blue]{hyperref}

\newtheorem{theorem}{Theorem}%[section]
\newtheorem{proposition}[theorem]{Proposition}
\newtheorem{lemma}[theorem]{Lemma}

\newtheorem{claim}[theorem]{Claim}

\newcommand{\R}{\mathbb{R}}
\newcommand{\N}{\mathbb{N}}

\newcommand{\C}{\mathbb{C}}

\newcommand{\Prob}[1]{\mathbb{P} \left( #1 \right) }

\newcommand{\abs}[1]{\left\vert #1 \right\vert}
\newcommand{\norme}[1]{\left\| #1 \right\| }
\newcommand{\scalar}[1]{\left\langle #1 \right\rangle }

\newcommand{\eps}{\varepsilon}

\renewcommand{\d}{\mathrm{d}}

\newcommand{\E}{\mathbb{E}}

\renewcommand{\P}{\mathbb{P}}

\newcommand{\reals}{\mathbb{R}}
\newcommand{\complexes}{\mathbb{C}}

\newcommand{\integers}{\mathbb{Z}}

\newcommand{\1}{\mathbf{1}}

\renewcommand{\S}{\mathcal{S}}
\DeclareMathOperator{\supp}{supp}
\DeclareMathOperator{\Span}{span}
\DeclareMathOperator{\Ker}{Ker}

\title{The density of imaginary multiplicative chaos is positive}
\author{Juhan Aru\thanks{École Polytechnique Fédérale de Lausanne, Lausanne, Switzerland; juhan.aru@epfl.ch} \and Antoine Jego\thanks{École Polytechnique Fédérale de Lausanne, Lausanne, Switzerland; antoine.jego@epfl.ch} \and Janne Junnila\thanks{University of Helsinki, Helsinki, Finland; janne.junnila@helsinki.fi}}

\date {}

\begin{document}

\renewcommand{\theparagraph}{\thesubsection.\arabic{paragraph}} %la numérotation des paragraphes ne prend pas en compte les subsubsections
\maketitle

\begin{abstract}
Consider a log-correlated Gaussian field $\Gamma$ and its associated imaginary multiplicative chaos $:e^{i \beta \Gamma}:$ where $\beta$ is a real parameter.
In \cite{aru2022density}, we showed that for any nonzero test function~$f$, the law of $\int f :e^{i \beta \Gamma}:$ possesses a smooth density with respect to Lebesgue measure on $\C$. In this note, we show that this density is strictly positive everywhere on $\C$. Our simple and direct strategy could be useful for studying other functionals on Gaussian spaces.
%In particular, it implies that the negative $p$-th moments of $\int f :e^{i \beta \Gamma}:$ diverge when $p \leq -2$.
\end{abstract}

\section{Introduction}

Let $\Gamma$ be a logarithmically correlated Gaussian field on some domain $U \subset \reals^d$ whose covariance kernel $C$ (formally $C(x,y) = \E[\Gamma(x)\Gamma(y)]$, $x,y \in U$) can be written as
\begin{equation}\label{eq:logkernel}
C(x,y) = \log \frac{1}{|x-y|} + g(x,y), \quad x,y \in U,
\end{equation}
where $g \in H^{d+\varepsilon}_{\mathrm{loc}}(U \times U) \cap L^2(U \times U)$ for some $\eps > 0$, %\footnote{For any $s \in \reals$ and $U \subset \reals^d$ we denote by $H^s_{\mathrm{loc}}(U)$ the space of distributions $f$ for which $\varphi f \in H^s(\reals^d)$ for all $\varphi \in C_c^\infty(U)$. \aj{Do we keep this footnote?}}
is symmetric ($g(x,y) = g(y,x)$) and bounded from above.
Throughout this article and as in \cite{aru2022density}, we will make the assumption that
\begin{equation}
\label{E:assumption}
    \star\text{ } \Gamma \text{ is nondegenerate in the sense that } C \text{ is an injective operator on } L^2(U).
\end{equation}
Let us now fix $\beta \in (0,\sqrt{d})$. For any $f \in L^\infty(U,\C)$ we may define the imaginary chaos $\mu$ tested against $f$ via the regularisation and renormalisation procedure
\[\mu(f) \coloneqq \lim_{\varepsilon \to 0} \int_U f(x) e^{i\beta \Gamma_\varepsilon(x) + \frac{\beta^2}{2} \E \Gamma_\varepsilon(x)^2} \, dx,\]
where $\Gamma_\varepsilon = \Gamma * \phi_\eps$ is a convolution approximation of $\Gamma$ against some smooth mollifier $\phi_\varepsilon = \eps^{-d} \phi(\cdot/\eps)$. The above limit takes place in $L^2$ and the resulting limiting random variable does not depend on the specific choice of mollifier \cite{JSW, complexGMC}. We will sometimes denote this random variable by $\int_U f : e^{i\beta \Gamma}:$, where $: e^{i\beta \Gamma}:$ stands for the Wick exponential of $i\beta \Gamma$.

In \cite{aru2022density} and under the above assumptions, we showed that for any nonzero $f \in C_c(U,\C)$, the law of $\mu(f)$ is absolutely continuous w.r.t. Lebesgue measure on $\C$ and the density is a Schwartz function\footnote{This result was stated for real-valued test functions in \cite{aru2022density} but extends readily to complex-valued test functions. See \cite[Proof of Theorem 3.2]{aru2024noise} for further explanations.}. The main result of the current paper shows that this density is everywhere positive:

\begin{theorem}\label{th:positivity}
Consider a nonzero test function $f \in C_c(U,\C)$.
Then for any $z_0 \in \C$, the limit
\begin{equation}
    \label{E:T_limit}
\lim_{r \to 0^+} r^{-2} \Prob{ |\mu(f) - z_0| < r}
\end{equation}
is strictly positive. In particular, the density of $\mu(f)$ is strictly positive everywhere.
\end{theorem}

Note that the existence of the limit \eqref{E:T_limit} follows from the existence of a density for $\mu(f)$. The contribution of the current paper is to show that it does not vanish.
As a direct consequence of the positivity of the density at the origin, we now have a complete understanding of the blow-up of the moments of $|\mu(f)|$:
\[
\forall p \in (-2,+\infty), \quad \E[|\mu(f)|^p] < +\infty
\qquad \text{and} \qquad
\forall p \in (-\infty, -2], \quad \E[|\mu(f)|^p] = +\infty.
\]

\medskip

A particular case of interest is the imaginary chaos corresponding to the Gaussian field $\Gamma$ on the circle $\S^1$ whose covariance is given by
\[
\E[\Gamma(e^{i\theta}) \Gamma(e^{i \theta'})]
= - \log |e^{i\theta} - e^{i\theta'}|, \quad \theta, \theta' \in [0,2\pi],
\]
together with the test function $f \equiv 1$, i.e. the total mass of the corresponding imaginary chaos. This case is not covered by the theorem above, however we explain in Section \ref{SS:circle} how one can modify our arguments to treat this case too. This log-correlated Gaussian field gives rise to an exactly solvable real chaos \cite{MR2430565, MR4047550}: the Fyodorov--Bouchaud (FB) formula is an explicit expression for all the moments of the total mass of the real chaos and in fact determines its law. Moreover, the analytic continuation of the moment-formula to $\gamma = i\beta$ yields finite negative moments up to $-2/\beta^2$.
Now, in \cite{aru2022density}, we showed that the analytic continuation of this formula from the real case to the imaginary chaos cannot in general correspond to the $-1$th moment of the imaginary chaos. In that argument we made use of the negative moments of the absolute value of the total mass. Here, our result implies that, for the absolute value, moments of order $p \leq -2$ blow up. Still, given the FB formula, one may wonder whether thanks to some cancellations it may or may not be possible to make sense of the negative moments $\E[\mu(f)^p]$ without the absolute values for some values of $p \leq -2$.

\medskip

Understanding the density of imaginary chaos is of importance in studying the properties of imaginary chaos itself \cite{aru2022density, aru2024noise} but, as just explained, also has implications for related objects like real multiplicative chaos \cite{aru2022density} and possibly also continuum limits of spin models \cite{JSW}.
More widely, the problem of proving existence and positivity of densities of functionals on Gaussian spaces can be put in the wider context of Malliavin calculus \cite{Nualart}. In particular, there are known conditions for obtaining positivity for Wiener functionals using Mallivan calculus, e.g. \cite{MalliavinNualart2009, bally2013positivity, arous1991annulation}. None of these, nor small modifications thereof seem to apply in our concrete setting, hence we propose a new, simple and direct general strategy that could potentially apply in other contexts too; see Section~\ref{SS:strategy} below. 

\medskip

As a possibly interesting side-result and recalling that one can make sense of $:e^{i\beta \Gamma}:$ as a random element of $H^{-d/2-\eps}(\R^d)$ for any $\eps>0$ \cite[Theorem 1.1]{JSW}, we prove that morally the support of imaginary chaos is all of $H^{-d/2-\eps}$:

\begin{proposition}\label{P:intermediate}
Let $\eps>0$, $f \in L^\infty(U,\C)$ not identically zero and $K \subseteq U$ a compact subset of $U$.
For any $\eta>0$, the probability
\[
\Prob{ \|1_K(f :e^{i \beta \Gamma}: - 1)\|_{H^{-d/2-\eps}(\R^{d})} \leq \eta }
\]
is strictly positive.
\end{proposition}

This note is structured as follows. Our general strategy is explained in the following subsection, and we deal with the rigorous set-up in Section \ref{S:setup}. Section \ref{S:proofs} contains the proofs of our main results. We will start in Section \ref{SS:deterministic} by proving a deterministic result which guarantees the existence, for any given function $f$, of some smooth oscillating function $e^{i \beta a(\cdot)}$ such that the integral $\int f e^{i \beta a}$ takes the desired value. We will then move to the proof of Theorem \ref{th:positivity} in Section \ref{SS:proof_T}, assuming Proposition~\ref{P:intermediate}. We will finally prove Proposition \ref{P:intermediate} in Section \ref{SS:proof_P}.

\paragraph*{Acknowledgements}
J.A. and A.J. are supported by Eccellenza grant 194648 of the Swiss National Science Foundation and are members of NCCR Swissmap. J.J. is supported by The Finnish Centre of Excellence (CoE) in Randomness and Structures and was a member of NCCR Swissmap. The authors thank the anonymous referees for their careful readings and suggestions.

\subsection{High level strategy}\label{SS:strategy}

Let $f \in C_c(U,\C)$ be a nonzero test function and $z_0 \in \C$. To bound from below the probability that $|\mu(f)-z_0| < r$, we will use the following high level strategy. This strategy could be proved useful in other contexts, especially in settings that can be studied with Malliavin calculus.
\begin{itemize}[leftmargin=*]
    \item We find an orthonormal basis $(h_n)_{n \geq 1}$ of the Cameron--Martin space $H_\Gamma$ of $\Gamma$ (see Section~\ref{S:setup} for the definition of $H_\Gamma$), which may depend on $z_0$ and $f$, and such that the following holds. Decomposing $\Gamma = \sum_{n \geq 1} A_n h_n$ where $A_n$, $n \geq 1$, are i.i.d. standard Gaussian random variables, we can view the random variable $\mu(f)$ as a function of $(A_n)_{n \geq 1}$:
    \[
    \mu(f) = \psi(A_n, n \geq 1), \qquad \text{with} \qquad \psi : \R^\N \longmapsto \C.
    \]
    \item We find $n_0 \geq 2$ and $a_1, \dots, a_{n_0} \in \R$ such that the map
    \[
    \varphi_0 : (u_1,u_2) \in \R^2 \longmapsto \E[ \psi(a_1 + u_1, a_2 + u_2, a_3, \dots, a_{n_0}, A_{n_0+1}, A_{n_0+2}, \dots ) ]
    \]
    satisfies $\varphi_0(0,0)=z_0$ and $\varphi_0: B \to \varphi_0(B)$ is a diffeomorphism for some neighbourhood $B$ of $(0,0)$.
    \item We show that the above properties are stable in the following sense. There exists an event $E \in \sigma(A_3, A_4, \dots)$ with positive probability which informally requires $A_n$ to be close to $a_n$ for $n=3, \dots, n_0$, and $A_n$ to have a typical behaviour for $n \geq n_0+1$ and such that the following holds. On the event $E$, the map
    \[
    \Phi : (u_1, u_2) \in \R^2 \longmapsto \psi(a_1+u_1,a_2+u_2, A_3, A_4, \dots)
    \]
    satisfies $\Phi(\mathbf{u})=z_0$ for some (random) $\mathbf{u} \in B/2$ and $\Phi: B \to \Phi(B)$ is a diffeomorphism. Moreover, on the event $E$, the determinant of the derivative map $D\Phi$ is uniformly bounded from above and below on the set $B$ by positive deterministic constants.
    \item We conclude by noticing that, on the event $E$, we have
    \[
    \Prob{|\mu(f)-z_0| < r \vert A_n, n \geq 3}
    = \Prob{ (A_1-a_1, A_2-a_2) \in \Phi^{-1}(B(z_0,r)) \vert A_n, n \geq 3} \geq c r^2.
    \]
\end{itemize}

For technical reasons, we actually use a slight variant of the above strategy but, roughly speaking, the above steps correspond to the following intermediate results.
Finding an orthonormal basis and real numbers $a_1, \dots, a_{n_0} \in \R$ such that $\varphi_0(0,0) = z_0$ is the content of the deterministic Lemma~\ref{L:existence_a}. The fact that $\varphi_0$ is a diffeomorphism on some ball centred at $(0,0)$ is proved in the proof of Lemma \ref{L:Phi}. The stability step is contained in (the proof of) Proposition~\ref{P:intermediate}.

\medskip

\subsection{Setup}\label{S:setup}

We recall some basic facts concerning the log-correlated Gaussian field $\Gamma$.
Note that its covariance operator $C$ defines a Hilbert--Schmidt operator on $L^2(U)$, and hence $C$ is self-adjoint and compact.
Since $C$ is positive definite, by the spectral theorem there exists a nonincreasing sequence of strictly positive eigenvalues $\lambda_1 \ge \lambda_2 \ge \dots > 0$ and corresponding orthogonal eigenfunctions $(f_k)_{k \ge 1}$ spanning the subspace $L \coloneqq (\Ker C)^\bot$ in $L^2(U)$ (which agrees with $L^2(U)$ under our assumption \eqref{E:assumption}). We may now construct the log-correlated field $\Gamma$ via its Karhunen--Lo\`eve expansion
\begin{equation}\label{eq:logfieldexpansion}
\Gamma = \sum_{k \ge 1} A_k C^{1/2} f_k = \sum_{k \ge 1} A_k \sqrt{\lambda_k} f_k,
\end{equation}
where $(A_k)_{k \ge 1}$ is an i.i.d. sequence of standard normal random variables. It has been shown in \cite[Proposition~2.3]{JSW} that the above series converges in $H^{-\varepsilon}(\reals^d)$ for any fixed $\varepsilon > 0$ (extending the relevant functions/field by 0 outside of $U$).

From the KL-expansion one can see that heuristically $\Gamma$ is a standard Gaussian on the space $H \coloneqq C^{1/2} L$. The space $H$ is called the \emph{Cameron--Martin space} of $\Gamma$, and it becomes a Hilbert space by endowing it with the inner product $\langle f, g \rangle_H = \langle C^{-1/2} f, C^{-1/2} g \rangle_{L^2}$, where $C^{-1/2} f, C^{-1/2} g \in L$. This definition makes sense since $C^{1/2}$ is an injection on $L$. %We will define the KL-basis $(e_k)_{k \ge 1}$ for $H$ by setting $e_k \coloneqq \sqrt{\lambda_k} f_k$, and we will also write $\langle \Gamma, h \rangle_H \coloneqq \sum_{k=1}^\infty A_k \langle h, e_k \rangle_H$ for $h \in H$. The left hand side in the latter definition is purely formal since $\Gamma \notin H$ almost surely.
Alternatively, the space $H$ is the space of distributions $f$ such that $\Gamma + f$ is absolutely continuous with respect to $\Gamma$; see for instance \cite[Section 1.9]{berestycki2024gaussian} in the case of the 2D GFF.

\medskip

We now record a lemma concerning Cameron--Martin spaces for ease of future reference.

\begin{lemma}\label{L:CM}
    Let $K \subset U$ be any compact subset of $U$. There exists a Gaussian field $\Gamma'$ defined on $\R^d$ such that $\Gamma \overset{(d)}{=} \Gamma'$ in $K$ and such that the Cameron--Martin space of $\Gamma'$ contains $C^\infty_c(\R^d)$.
\end{lemma}

We would like to point out that it can be the case that two fields $\Gamma$ and $\Gamma'$ have the same law when restricted to $K$, but the  subsets $\{ f \in H_\Gamma : \supp f \subset K \}$ and $\{ f \in H_{\Gamma'} : \supp f \subset K \}$ of the Cameron--Martin spaces of $\Gamma$ and $\Gamma'$ do not agree. For an example consider two independent standard Gaussians $X$ and $Y$, the fields $(X, Y)$ and $(X,X)$ and as $K$ the first coordinate. In the case of $(X,Y)$, the whole Cameron--Martin space is spanned by $(1,0)$ and $(0,1)$, so its subset of ``functions'' supported by the first coordinate is simply $\R$. In the second case, the Cameron--Martin space is spanned by $(1,1)$ so the only ``function'' with support in the first coordinate is 0.
Lemma~\ref{L:CM} was stated in this way because of this counter-intuitive property.

\begin{proof}
By \cite[Theorem 4.5]{aru2022density} and because $\Gamma$ is nondegenerate (see \eqref{E:assumption}), we may decompose in $K$, $\Gamma = L+R$
as the sum of a Hölder continuous field $R$ and an independent almost $\star$-scale invariant field $L$ whose covariance equals
\[\E[L(x) L(y)] = \int_{0}^\infty k(e^u(x-y))(1 - e^{-\delta u}) \, du.\]
Here, $\delta>0$ is a parameter and $k : \R^d \to [0,\infty)$ is a rotationally symmetric seed covariance with $k(0)=1$, $\supp k \subset B(0,1)$ and such that
\begin{equation}\label{E:seed_k}
    \exists s > \frac{d+1}{2}, \quad \forall \xi \in \R^d, \quad 0 \leq \hat{k}(\xi) \leq (1+ |\xi|^2)^{-s}.
\end{equation}
By \cite[Lemma 4.8]{aru2022density}, the Cameron--Martin space of $L$ contains $C^\infty_c(\R^d)$. By \cite[Lemma~4.1]{aru2022density}, this implies that the Cameron--Martin space of $L+R$ also contains $C^\infty_c(\R^d)$ concluding the proof.
\end{proof}

\section{Proofs}\label{S:proofs}

\subsection{A deterministic result}\label{SS:deterministic}

\begin{lemma}\label{L:existence_a}
Let $f \in L^1(\R^d,\C)$ not identically zero. For all $z_0 \in \C$ with $|z_0| < \norme{f}_{L^1}$, 
there exists a function $a \in C^\infty(\R^d,\R)$ %which is not constant in the support of $f$ 
such that
\begin{equation}
    \label{E:L_existence_a}
\int f(x) e^{i\beta a(x)} \, dx = z_0
\end{equation}
and such that $\beta a + \arg(f)$ is not constant in the support of $f$.
\end{lemma}

We emphasise that the function $a$ is required to be smooth.
Of course, if $f$ has compact support within some open domain $U$, then one can also require $a$ to have compact support in $U$.

In general, this result cannot be extended to $|z_0| \geq \norme{f}_{L^1}$. Indeed, by the triangle inequality, one cannot find such a function if $|z_0| > \norme{f}_{L^1}$. If $|z_0| = \norme{f}_{L^1}$ and if \eqref{E:L_existence_a} holds for some function $a$, then $f e^{i \beta a}$ agrees with $|f| z_0/|z_0|$ Lebesgue-almost everywhere. In general this equality cannot be achieved by a continuous function $a$.

\begin{proof}
Let us first assume the existence of such a function $a_0$ when $z_0 = 0$ (this is actually the difficult part of the proof). Let $z_0 \in \C$ with $0 < |z_0| < \norme{f}_{L^1}$.
Let $v = -\arg(f)/\beta$ where we arbitrarily define $\arg(f)=0$ when $f=0$.
%\[
%v : x \in \R^d \longmapsto \left\{ \begin{array}{cc}
%    0 & \text{if} \quad f(x) \geq 0 \\
%    \pi/\beta & \text{if} \quad f(x) < 0. 
%\end{array} \right.
%\]
By construction, $f e^{i \beta v} = |f|$. Let $\rho_\eps = \eps^{-d} \rho(\cdot/\eps)$, $\eps \geq 0$, be a sequence of smooth mollifiers. Define $b_\eps = (1-\eps) v * \rho_\eps + \eps a_0$. By dominated convergence theorem, the map
\[
\eps \in [0,1] \longmapsto \abs{ \int f(x) e^{i \beta b_\eps(x)} \, dx}
\]
is continuous, equal to $\norme{f}_{L^1}$ when $\eps =0$ and equal to 0 when $\eps =1$. Therefore, there exists some $\eps \in (0,1)$ such that the integral of $f e^{i \beta b_\eps}$ has the same modulus as $z_0$. 
Setting $a \in C^\infty(\R^d,\R)$ to be equal to $b_{\eps}(\cdot) + \theta$ on the support of $f$, where $\theta \in \R$ appropriately changes the phase, we obtain that
 $\int f e^{i \beta a} = z_{0}.$

\begin{comment}
Let us first show that it is enough to prove the claim for one value of $t$. We thus first assume that there exist $t_0 \in \R$ and a function $a_{t_0} \in C^\infty_c(U,\R)$, not constant in the support of $f$ and such that \eqref{E:L_existence_a} holds with $t_0$ instead of $t$. Define for all $\lambda \geq 0$, $b_\lambda : x \in U \mapsto a_{t_0}(x) + \lambda x$. The map
\[
\lambda \in [0,\infty) \longmapsto \abs{ \int f(x) e^{i\beta b_\lambda(x)} \, dx }
\]
is continuous, equal to $e^{-t_0}|z_0|$ at 0 and goes to zero as $\lambda \to \infty$ (which essentially follows by integration by parts). Therefore, for all $t \geq t_0$, there exists $\lambda = \lambda(t) \geq 0$ such that the modulus of $\int f e^{i \beta b_\lambda}$ equals $e^{-t} |z_0|$. 
Setting $a_t(\cdot)$ with compact support to be equal to $b_{\lambda}(\cdot) + \theta$ on the support of $f$, where $\theta =\theta(t) \in \R$ appropriately changes the phase, we obtain that for all $t \geq t_0$,
 \[\int f(x) e^{i \beta a_t(x) + t} \, dx = z_{0}.\]
It remains to prove existence of one such $t_0$.

When $z_0 \neq 0$, the proof is fairly straightforward. Indeed, note first that since $f$ is not identically zero, we can find a function $a_0 \in C^\infty_c(U, \R)$ which is not constant in the support of $f$ such that $\int f e^{i a_0} \neq 0$. 
If $z_0 \neq 0$, we then fix some $t \in \R$ such that $\abs{ \int f e^{i a_0} } e^ t = |z_0|$.
We then proceed as before to change appropriately the phase.
\end{comment}

We now treat the delicate case of $z_0 = 0$.
We will first find a continuous function $\tilde{a}$ such that $\int |f| e^{i \beta \tilde{a}} = 0$.
The function $\hat{a} := \arg(f)/\beta + \tilde a$ will thus satisfy $\int f e^{i \beta \hat{a}} = 0$ but will not have the desired regularity.
We will then show that there is a smooth perturbation $a$ of $\hat{a}$ which has the same property.
%Let $J \coloneqq \int f(x) \, dx$. If $J \neq 0$,
We define
\[\tilde{a}(x_1,\dots,x_d) = b(x_d) = \frac{2\pi}{\beta \norme{f}_{L^1}} \int_{-\infty}^{x_d} \int_{\reals^{d-1}} |f(u_1,\dots,u_{d-1},u_d)| \, d u_1 \dots d u_d.\]
%If $J = 0$, we define $\tilde a$ similarly except that the multiplicative factor is simply $2\pi /\beta$.
We have
\[
\lim_{x_d \to +\infty} b(x_d) = 2\pi/\beta
\quad \text{and} \quad \lim_{x_d \to -\infty} b(x_d) = 0.
\]
%\[
%\lim_{x_d \to +\infty} b(x_d) = \left\{ \begin{array}{cc}
%    2\pi/\beta & \text{if} \quad J \neq 0 \\
%    0 & \text{if} \quad J = 0
%\end{array} \right.
%\quad \text{and} \quad \lim_{x_d \to -\infty} b(x_d) = 0.
%\]
Hence, %when $J \neq 0$,
\[\int_{\reals^d} |f(x)| e^{i \beta \tilde{a}(x)} \, dx = \frac{\beta \norme{f}_{L^1}}{2\pi} \int_{-\infty}^\infty b'(x_d) e^{i \beta b(x_d)} \, dx_d = \frac{\beta \norme{f}_{L^1}}{2 \pi} \cdot \frac{e^{2 \pi i} - 1}{i\beta} = 0\]
%Similarly, when $J = 0$, $\int f e^{i \beta \tilde a}=0$
as desired.

To define the perturbation $a$ of $\hat a$, notice first that, since the phase of $f e^{i \beta \hat a}$ is not constant, there exist smooth functions $g_1, g_2 \in C^\infty(\R^d,\R)$ such that
\[\int f(x) g_j(x) e^{i\beta \hat{a}(x)} = \theta_j, \quad j=1,2,\]
where $\theta_1,\theta_2$ are two complex numbers with modulus one which are not linearly correlated.
We will make an ansatz for $a$ of the form
\[a(x) = (\rho_{\varepsilon} * \hat{a})(x) + s_1 g_1(x) + s_2 g_2(x),\]
where $\rho_{\varepsilon} = \varepsilon^{-d} \rho(\cdot/\varepsilon)$ is a sequence of standard smooth mollifiers and $s_j \in \reals$, $j = 1,2$, are some parameters.
Let us now look at the family of smooth maps $(\eta_{\varepsilon})_{\eps\geq 0}$ given by
\[\eta_{\varepsilon} : \mathbf{s} = (s_1,s_2) \in \R^2 \longmapsto \int f(x) e^{i\beta (\rho_{\varepsilon} * \hat{a})(x) + s_1 g_1(x) + s_2 g_2(x)} \, dx \in \C,\]
where we used the convention that $\rho_0 * \hat{a} = \hat{a}$. Note that $\eta_\eps \to \eta_0$ uniformly on $\R^2$ by elementary inequalities and the convergence $\rho_\eps * \hat{a} \mathbb{1}_{f \ne 0} \to \hat{a} \mathbb{1}_{f \ne 0}$ in $L^1(\R^d)$.
To conclude the proof it is enough to find some small $\varepsilon > 0$ and $\mathbf{s} \in \reals^2$ such that $\eta_\varepsilon(\mathbf{s}) = 0$.
Note that $\eta_0$ is a map with
$
\eta_0(0,0) = 0, \partial_1 \eta_0(0,0) = \theta_1, \partial_2 \eta_0(0,0) = \theta_2.
$
In particular, $D\eta_0(0,0)$ is invertible and, by the inverse function theorem, there exists a neighbourhood $B(0,\tilde{R})$ of $(0,0)$ such that $\eta_0$ is invertible with inverse $h \colon \eta_0(B(0,\tilde{R})) \to \reals^2$.
Let us fix $R > 0$ such that $B(0,R) \subset \eta_0(B(0,\tilde{R}))$ and consider $r > 0$ for which $\eta_0(\overline{B(0,r)}) \subset B(0,R/2)$.
We next note the following simple corollary of Brouwer fixed-point theorem:

\begin{lemma}\label{L:Brouwer}
Let $F : \overline{B(0,r)} \to \reals^2$ be a continuous function such that $|F(\mathbf{s}) - \mathbf{s}| \le r$ for all $\mathbf{s} \in \overline{B(0,r)}$.
Then $F(s) = 0$ for some $\mathbf{s} \in \overline{B(0,r)}$.
\end{lemma}

\begin{proof}[Proof of Lemma \ref{L:Brouwer}]
    Even though it is likely that this is a classical result, we provide a proof for completeness. Consider the map $G : \mathbf{s} \in \overline{B(0,r)} \mapsto -F(\mathbf{s}) + \mathbf{s}.$
    By assumption, $G(\overline{B(0,r)}) \subset \overline{B(0,r)}$ and $G$ is continuous. So Brouwer fixed-point theorem applies giving the existence of a fixed point for $G$, or equivalently a point $\mathbf{s}$ such that $F(\mathbf{s}) = 0$.
\end{proof}

To apply the lemma, we consider the functions $h \circ \eta_\varepsilon$ on $\overline{B(0,r)}$.
For $\varepsilon$ small enough we see that they are well-defined since then $|\eta_\varepsilon(\mathbf{s})| \le 3R/4$ for all $\mathbf{s} \in B(0,r)$.
Moreover since $h$ is Lipschitz in $B(3R/4)$, we have that $|h(\eta_\varepsilon(\mathbf{s})) - \mathbf{s}| \le L |\eta_\varepsilon(\mathbf{s}) - \eta_0(\mathbf{s})|$ for some constant $L > 0$.
Thus for small enough $\varepsilon$ the right hand side is less than $r$ and the lemma above applies and we have that
$h(\eta_\varepsilon(\mathbf{s})) = 0$ for some $\mathbf{s} \in B(0,r)$, implying that $\eta_\varepsilon(\mathbf{s}) = 0$ and in particular yielding the required $a$.
This concludes the proof of Lemma \ref{L:existence_a}.
\end{proof}

\subsection{Proof of Theorem \ref{th:positivity} assuming Proposition \ref{P:intermediate}}\label{SS:proof_T}

We start with a key intermediate result.

\begin{lemma}\label{L:Phi}
Let $f \in C_c(U,\C)$ not identically zero and $z_0 \in \C$.
By changing $\Gamma$ by a field which agrees with $\Gamma$ on the support of $f$ if necessary, there exist an orthonormal basis $(h_{n})_{n=1}^{\infty}$ of the Cameron--Martin space of $\Gamma$, a constant $C>1$, $a_1, a_2 \in \R$ and a ball $B=B(0,\delta) \subset \R^2$ such that the following holds. Let $A_n = \scalar{\Gamma,h_n}_H, n \geq 1$.
On an event $E \in \sigma(A_n, n \geq 3)$ with positive probability,
the random smooth map
\begin{equation}
    \label{E:map_Phi}
\Phi : (u_{1},u_{2}) \longmapsto \int f(x) e^{i \beta ((a_1+u_{1}) h_{1}(x) + (a_2+u_{2}) h_{2}(x)) + \frac{\beta^2}{2} (h_1(x)^2 + h_2(x)^2) }  :e^{i\beta\sum_{n \geq 3} A_n h_n(x)}: dx
\end{equation}
is a diffeomorphism $B \to \Phi(B)$. Moreover, on the event $E$, $\Phi(\mathbf{u}) = z_{0}$ for some random $\mathbf{u} \in B(0,\delta/2)$, and for all $\mathbf{u} \in B$ the derivative map $D \varphi \colon T_{\mathbf{u}} \reals^{2} \to T_{\Phi(\mathbf{u})} \complexes$ is bounded in norm by $C$ and $1/C$ from above and away from $0$ respectively.
\end{lemma}

\begin{comment}
\begin{lemma}\label{lem:positivity1}
There exist $t \in \R$, an orthonormal basis $(h_{n})_{n=1}^{\infty}$ of the Cameron--Martin space of $\Gamma$, a constant $C>1$ and a ball $B = B(x_{0},\delta) \subset \reals^{2}$ such that the following holds. Consider $A_3 = \scalar{\Gamma,h_3}_{H}$. On an event $E \in \sigma(A_3)$ with positive probability, the random smooth map
\begin{equation}
    \label{E:map_phi}
\varphi : (u_{1},u_{2}) \in B \longmapsto \int f(x) e^{i \beta (u_{1} h_{1}(x) + u_{2} h_{2}(x) + A_3 h_{3}(x))+t} \; dx \in \C
\end{equation}
is a diffeomorphism $B \to \varphi(B)$. Moreover, on the event $E$, $\varphi(x) = z_{0}$ for some random $x \in B(x_0,\delta/2)$, and for all $x \in B$ the derivative map $D \varphi \colon T_{x} \reals^{2} \to T_{\varphi(x)} \complexes$ is bounded in norm by $C$ and $1/C$ from above and away from $0$ respectively.
\end{lemma}
\end{comment}

\begin{proof}[Proof of Lemma \ref{L:Phi}, assuming Proposition \ref{P:intermediate}]
By Lemma \ref{L:existence_a}, there exist $t \in \R$ and $g_3 \in C^\infty_c(U,\R)$ such that $\arg f + \beta g_3$ is not constant in the support of $f$ and such that
\begin{equation}
\label{E:z0}
\int f(x) e^{i\beta g_3(x) + t} \, dx = z_0.
\end{equation}
Since the phase of $f e^{i \beta g_3}$ is not constant, we can pick two functions $g_1$ and $g_2 \in C_c^\infty(U,\R)$ such that 
\begin{equation}
    \label{E:thetaj}
\int f(x) g_j(x) e^{i \beta g_3(x) +  t} dx = \theta_j, \quad j = 1,2,
\end{equation}
where $\theta_1$ and $\theta_2$ are two complex numbers with modulus one which are not linearly correlated.
Let $V \subset U$ be an open subset of $U$ containing the support of $f$. We can modify the definitions of the functions $g_1, g_2$ and $g_3$ inside $U \setminus V$ without changing the value of the integrals \eqref{E:z0} and \eqref{E:thetaj}. We will in particular assume the existence of three disjoint open sets $W_1, W_2, W_3 \subset U \setminus V$ such that for all $i \neq j$,
\begin{equation}
    \label{E:Wj}
    g_j \mathbf{1}_{W_j} \in C^\infty_c(W_j,\R), \quad \norme{g_j \mathbf{1}_{W_j}}_H > 0 \quad \text{and} \quad g_j \mathbf{1}_{W_i} =0.
\end{equation}
These conditions in particular imply that $g_1, g_2$ and $g_3$ are linearly independent, but will be also useful at the end of the current proof for our ``restriction trick''.

Let us now fix an ON-basis $(h_{n})_{n=1}^{\infty}$ of $H$ with the first two elements $h_{1},h_{2}$ spanning $\Span\{g_{1},g_{2}\}$ and the first three elements $h_1, h_2, h_3$ spanning $\Span\{g_1, g_2, g_3\}$. This can be done thanks to Lemma~\ref{L:CM} (without loss of generality, we may consider another field $\Gamma'$ that agree with $\Gamma$ on the support of $f$ and whose Cameron--Martin space contains $C^\infty_c(\R^d)$).
Because $g_3 \in \Span \{h_1,h_2,h_3\}$, there exist $a_1,a_2,a_3 \in \R$ such that $g_3 = a_1h_1 + a_2h_2 + a_3h_3$.
The map
\[
\varphi_{0} : (u_{1},u_{2}) \in \R^2 \longmapsto \int f(x) e^{i \beta ((a_1 + u_{1}) h_{1}(x) + (a_2+u_{2}) h_{2}(x) + a_3h_3(x)) +  t} \, dx \in \C.
\]
satisfies
\[
\varphi_0(0,0) = z_0
\quad \text{and} \quad
\partial_j \varphi_0(0,0) = i\beta \int f(x) h_j(x) e^{i \beta g_3(x) + t} \, dx, \quad j=1,2.
\]
Because $\theta_1$ and $\theta_2$ are not linearly correlated (see \eqref{E:thetaj}),
$D \varphi_0(0,0)$ is invertible. By the inverse function theorem, there exists a ball $B=B(0,2\delta)$ such that $\varphi_0 : B \to \varphi_0(B)$ is a diffeomorphism. Our main task now is to show that the same is true when we perturb $\varphi_0$ and consider $\Phi$ instead.
It is enough to show that for all $\eta >0$ arbitrarily small, the probability
\begin{equation}\label{E:eta}
    \P( \norme{\Phi-\varphi_0}_{L^\infty(B)} < \eta, \norme{\partial_j \Phi-\partial_j \varphi_0}_{L^\infty(B)} < \eta, j=1,2) > 0.
\end{equation}
We can bound $|\Phi(u_1,u_2) - \varphi_0(u_1,u_2)|$ by
\begin{align*}
    & \abs{ \int f e^{i \beta((a_1+u_1)h_1 + (a_2+u_2)h_2) + t} (e^{i\beta A_3 h_3} - e^{i\beta a_3 h_3})} \\
    & + \abs{ \int f e^{i \beta((a_1+u_1)h_1 + (a_2+u_2)h_2 + A_3h_3)} \Big(e^{\frac{\beta^2}{2}(h_1^2+h_2^2+h_3^2)}:e^{i\beta \sum_{n \geq 4} A_n h_n}: - e^t \Big) } \\
    & \leq \norme{f}_{L^1(U)} e^t \norme{e^{i\beta A_3h} - e^{i\beta a_3h}}_{L^\infty(U)} + \norme{f e^{i \beta((a_1+u_1)h_1 + (a_2+u_2)h_2+A_3h_3)}}_{H^d(\R^d)} \times \\
    & \hspace{150pt} \times \norme{\big(e^{\frac{\beta^2}{2}(h_1^2+h_2^2+h_3^2)}:e^{i\beta \sum_{n \geq 4} A_n h_n}: - e^t \big)\mathbf{1}_V}_{H^{-d}(\R^d)} .
\end{align*}
Similarly, for $j=1,2$, $|\partial_j \Phi(u_1,u_2) - \partial_j \varphi_0(u_1,u_2)|$ is also bounded by the above sum of two terms.
The first term can be made arbitrarily small with positive probability by making $A_3$ close to $a_3$.
Assuming for a moment that we can apply Proposition \ref{P:intermediate} to the field $\sum_{n \geq 4} A_n h_n$ in $V$ (i.e. that it satisfies the assumption \eqref{E:assumption}), we can make the second term arbitrarily small with positive probability conditionally on $A_3$.
This would then show \eqref{E:eta} and conclude the proof of Lemma \ref{L:Phi}.

\textbf{Restriction trick.}
It remains to prove that the field $\sum_{n \geq 4} A_n h_n$ restricted to $V$ satisfies the assumption \eqref{E:assumption} (notice that it would not be the case without the restriction to a subset $V$).
Let
\[
\widetilde C : (x,y) \in V \times V \mapsto \sum_{n \geq 4} h_n(x) h_n(y).
\]
We want to show that $\widetilde C$ is injective on $L^2(V)$.
Let $\tilde f \in L^2(V)$ be such that $\widetilde C \tilde f = 0$. We want to show that $\tilde f = 0$.
Recall the existence of the functions $g_1, g_2, g_3$ and the subsets $W, W_1, W_2 \subset U \setminus V$ introduced in \eqref{E:Wj} and above. Let $j \in \{1,2,3\}$.
Because $\{h_n, n \geq 4\}$ does not span $L^2(W_j)$, there exists $\tilde f_j \in L^2(U)$ vanishing outside of $W_j$ such that $\int \tilde f_j h_n = 0$ for $n \geq 4$ and $\int \tilde f_j g_j \neq 0$. Let $f \in L^2(U)$ be defined as
\[
f = \tilde f \mathbf{1}_V + \sum_{j=1}^3 \lambda_j \tilde f_j,
\quad \text{where} \quad
\lambda_j = - \int \tilde f g_j \Big/ \int \tilde f_j g_j.
\]
Since $\widetilde C \tilde f = 0$,
we have
\begin{align*}
    C f = \sum_{n=1}^3 \Big( \int \tilde f h_n \Big) h_n + \sum_{n=1}^3 \sum_{j=1}^3 \lambda_j \Big( \int \tilde f_j h_n \Big) h_n.
\end{align*}
By definition of $\lambda_j$ and because $g_n$ vanishes on the support of $\tilde f_j$ for $n \in \{1,2,3\} \setminus \{j\}$, we have for all $n =1,2,3$,
\[
\sum_{j=1}^3 \lambda_j \Big( \int \tilde f_j g_n \Big) = - \int \tilde f g_n.
\]
Since $h_1, h_2$ and $h_3$ are linear combinations of $g_1, g_2$ and $g_3$, we deduce that for all $n=1,2,3,$
\[
\sum_{j=1}^3 \lambda_j \Big( \int \tilde f_j h_n \Big) = - \int \tilde f h_n,
\]
implying that $Cf = 0$. By assumption \eqref{E:assumption}, $C$ is injective on $L^2(U)$ and thus $f = 0$. We deduce that $\tilde f = f_{\vert V} = 0$ as desired. This concludes the proof.
\end{proof}

We can now prove:

\begin{proof}[Proof of Theorem \ref{th:positivity}, assuming Proposition \ref{P:intermediate}]
Let $z_0 \in \C$.
By definition \eqref{E:map_Phi} of the map $\Phi$, $\mu(f) = \Phi(A_1-a_1,A_2-a_2)$ a.s. By Lemma \ref{L:Phi}, on the event $E$, if $r>0$ is small enough,
\begin{align*}
  & \P ( \mu(f) \in B(z_{0},r) \vert A_n, n \geq 3 )
  = \P( (A_{1}-a_1,A_{2}-a_2) \in \Phi^{-1}(B(z_{0},r)) | A_n, n \geq 3 ).
\end{align*}
By Lemma \ref{L:Phi} and on the event $E$, the Lebesgue measure of $\Phi^{-1}(B(z_{0},r))$ is bounded from below by $c r^2$ for some deterministic constant $c>0$. It follows that for $r>0$ small enough and on the event $E$,
\begin{equation}
    \P ( \mu(f) \in B(z_0,r) \vert A_n, n \geq 3 ) \geq c r^2.
\end{equation}
We deduce that $\P ( \mu(f) \in B(z_0,r)) \geq c \P(E) r^2$ proving that the limit in \eqref{E:T_limit} is positive.
\end{proof}

\subsection{Proof of Proposition \ref{P:intermediate}}\label{SS:proof_P}

We now turn to the proof of Proposition \ref{P:intermediate}.

\begin{proof}[Proof of Proposition \ref{P:intermediate}]
Let $f: U \to \C$ and $K \Subset U$ be as in the statement of the proposition.
By \cite[Theorem 4.5]{aru2022density} and because $\Gamma$ is nondegenerate (see \eqref{E:assumption}), we may decompose in the compact $K$, $\Gamma = L+R$
as the sum of a Hölder continuous field $R$ and an independent almost $\star$-scale invariant field $L$ whose covariance equals
\[\E[L(x) L(y)] = \int_{0}^\infty k(e^u(x-y))(1 - e^{-\delta u}) \, du.\]
Here, $\delta>0$ is a parameter and $k : \R^d \to \R$ is a seed covariance satisfying the same assumptions as above \eqref{E:seed_k}.
The process $L$ is the limit of the smooth processes $L_t$ as $t \to \infty$ where $L_t$ has the covariance structure
\[\E[L_t(x) L_t(y)] = \int_{0}^t k(e^u(x-y))(1 - e^{-\delta u}) \, du.\]
We now define the approximations, for $T>0$,
\[ \Gamma_T = L_T + R,  \quad \text{and} \quad \Gamma_{T,\infty} = \Gamma - \Gamma_T\]
for the field $\Gamma$ and its tail.

We may bound
\begin{align}
\label{eq:proof_lem_positivity2}
 & \|1_K(f:e^{i \beta \Gamma}: - 1)\|_{H^{-d/2-\eps}(\reals^{d})} \leq X_T+Y_T
\end{align}
where
\[X_T = \|1_K(f:e^{i \beta \Gamma_{T}(x)}: - 1) :e^{i \beta \Gamma_{T,\infty}}:\|_{H^{-d/2-\eps}(\reals^{d})}\]
and
\[Y_T = \|1_K(:e^{i \beta \Gamma_{T,\infty}(x)}: - 1)\|_{H^{-d/2-\eps}(\reals^{d})}.\]
We now deal with $X_T$ and $Y_T$ separately. We first claim that
$\E[Y_T^2 ] \to 0$ as $T \to \infty$.
To show this, we compute
  \begin{align*}
    & \E[Y_T^2 ] = \E[\|\mathbf{1}_K(:e^{i \beta \Gamma_{T,\infty}(\cdot)}: - 1)\|_{H^{-d/2-\eps}(\reals^{d})}^2] \\
    & = \E \int_{\R^d} d\xi (1 + |\xi|^{2})^{-d/2-\eps} \int _{K \times K} dx \, dy (:e^{i\beta \Gamma_{T,\infty}(x)}: - 1) (:e^{-i \beta \Gamma_{T,\infty}(y)}: - 1) e^{-2\pi i \xi \cdot (x-y)} \\
    & = \int_{\R^d} d\xi (1 + |\xi|^{2})^{-d/2-\eps} \int_{K \times K} dx \, dy (e^{\beta^{2} \int_{T}^{\infty} k(e^{s} (x-y)) (1-e^{-\delta s}) \, ds} - 1) e^{-2\pi i \xi \cdot (x-y)}.
  \end{align*}
  Denote by $u_T(\xi)$ the above integral over $K \times K$.
  By dominated convergence theorem, for all $\xi \in \R^d$, $u_T(\xi) \to 0$ as $T \to \infty$.
  Since
  \[
  \sup_{T \geq 0} |u_T(\xi)| \leq \int_{K \times K} dx \, dy (e^{\beta^{2} \int_{0}^{\infty} k(e^{s} (x-y)) (1-e^{-\delta s}) \, ds} - 1) < \infty,
  \]
  we can conclude by dominated convergence theorem that $\E[Y_T^2] \to 0$ as $T \to \infty$ as claimed.
In the rest of the proof, we will pick
$T>0$ large enough so that $\E[Y_T] \leq \eta/4$.

It remains to deal with $X_T$.
By Lemma \ref{L:existence_a}, there exists $t_0 \in \R$ such that for all $t \geq t_0$, there exists $a_t \in C_c^\infty(U,\R)$, such that
\[
\int f e^{\beta^2 \E[R^2]/2} e^{i \beta a_t + t} = 1.
\]
$\E[L_T(x)^2]$ does not depend on $x$ and goes to infinity as $T \to \infty$. We thus also pick $T$ large enough so that $\E[L_T(x)^2]$ exceeds the above value of $t_0$. We can now find
$a \in C^\infty_c(U,\R)$ such that
\[
\int f e^{i \beta a + \beta^2 \E[\Gamma_T^2]/2} = 1.
\]
Since the Cameron--Martin space of $\Gamma_T$ contains $C_c^\infty(U)$ (see the proof of \cite[Lemma 4.8]{aru2022density} for details),
$\Gamma_T$ can be made arbitrarily close to $a$ with positive probability in say $H^d(\R^d)$-norm and $f :e^{i \beta \Gamma_T}: -1$ can be made arbitrarily close to 0. 
Moreover, a computation similar to the computation of $\E[Y_T^2]$ shows that $\E[ X_T^2 \vert \Gamma_{T} ]$ is controlled by $\|f :e^{i \beta \Gamma_T}: -1\|_{H^d(\R^d)}$. Altogether,
this shows that the probability of the event
\[
E := \{
\E[ X_T \vert \Gamma_{T} ] \leq \eta /4 \}
\]
is positive. Wrapping up and by Markov's inequality, we have
\begin{align*}
    \Prob{X_T+Y_T \leq \eta}
    & \ge \Prob{X_T+Y_T \leq \eta, E}
    = \Prob{E} - \Prob{X_T+Y_T > \eta, E} \\
    & \ge \Prob{E} - \E[ \E[X_T+Y_T\vert \Gamma_T] \mathbf{1}_E] / \eta.
\end{align*}
Since $\E[Y_T] \leq \eta/4$ and, on the event $E$, $\E[X_T \vert \Gamma_T] \leq \eta/4$, we have shown that 
\[
\Prob{X_T+Y_T \leq \eta} \geq \P(E)/2 >0.
\]
Together with \eqref{eq:proof_lem_positivity2}, this concludes the proof of Proposition \ref{P:intermediate}.
\end{proof}

\subsection{GFF on the circle}\label{SS:circle}

In this section we briefly explain how one can modify some of the arguments in the proof of Theorem~\ref{th:positivity} in order to be able to treat the case of the total mass of the multiplicative chaos associated to the GFF on the circle.
This field can be explicitly decomposed as
\[
\Gamma(e^{i\theta}) = \sum_{k \geq 1} A_k \frac{\sin(k \theta)}{\sqrt{k}} + B_k \frac{\cos(k \theta)}{\sqrt{k}},
\]
where $A_k, B_k, k \geq 1,$ are i.i.d. standard normal Gaussians.
Because the average of the GFF on the circle vanishes, the resulting field is not nondegenerate in the sense of \eqref{E:assumption}.
However, since the underlying structure is explicit, we can make the appropriate changes.
The following result is the main deterministic result we use instead of Lemma \ref{L:existence_a}.

\begin{lemma}\label{L:sin}
    There exists two open sets $O_1 \subset \R^2, O_2 \subset \C$ with $O_2$ containing the origin such that the map
    \[
    F : (s_1,s_2) \in O_1 \longmapsto \int_0^{2\pi} e^{i (s_1 \sin (\theta) + s_2 \cos(2\theta) )} \, \d \theta
    \]
    is a diffeomorphism from $O_1$ onto $O_2$.
\end{lemma}

\begin{proof}
    For $n \geq 0$, we will denote by $J_n$ the $n$-th Bessel function of the first kind.
    When $s_2 = 0$, $F(s_1,s_2)$ is explicit and is equal to $2\pi J_0(|s|)$. Let $j_0>0$ be the smallest positive root of $J_0$. One can show that, when $(s_1,s_2) = (j_0,0)$,
    \begin{align*}
        \frac{\partial F(s_1,s_2)}{\partial s_1} = 2 i \pi J_1(j_0)
        \quad \text{and} \quad
        \frac{\partial F(s_1,s_2)}{\partial s_2} = 2 \pi J_2(j_0).
    \end{align*}
    Since $J_1(j_0)$ and $J_2(j_0)$ do not vanish (this is a general fact concerning Bessel functions: the zeros of $J_n$ and $J_m$ are distinct when $n \neq m$), this shows that the determinant of $DF(j_0,0)$ does not vanish. We then conclude by the inverse function theorem.
\end{proof}

We now follow our strategy described in Section \ref{SS:strategy}. We fix $z_0 \in \C$. Let $(h_n)_{n \geq 1}$ be the orthonormal basis of the Cameron--Martin space composed of the functions $\theta \mapsto k^{-1/2} \sin(k \theta)$, $\theta \mapsto k^{-1/2} \cos(k \theta)$, $k \geq 1$.
We order these functions so that $h_1 = \sin(\cdot)$ and $h_2 = 2^{-1/2} \cos(2 \cdot)$. 
Decomposing the GFF on the circle as $\sum A_n h_n$, where $A_n$, $n \geq 1$, are i.i.d. standard Gaussian random variables, we view the total mass of the imaginary chaos as a function of $(A_n)_{n \geq 1}$:
\[
\int_{\S^1} :e^{i \beta \Gamma}: \quad = \quad \psi(A_1, A_2, \dots).
\]
Let $n_0 \geq 1$ be large. For $n=3, \dots, n_0$, let $a_n = 0$ and let $\varphi_0$ be the map
\[
(s_1,s_2) \in \R^2 \longmapsto \E[ \psi(s_1, s_2, a_3, \dots, a_{n_0}, A_{n_0+1}, A_{n_0+2}, \dots) ] = \int_{\S^1} e^{\frac{\beta^2}{2} \sum_1^{n_0} h_n^2} e^{i \beta(s_1 h_1+s_2h_2)}.
\]
Let $K_{n_0} = \frac{1}{2\pi} \int_{\S^1} e^{\frac{\beta^2}{2} \sum_1^{n_0} h_n^2}$.
We take $n_0$ large enough to ensure that:
\begin{itemize}[leftmargin=*]
    \item $K_{n_0} O_2$ contains the ball $B(0,2|z_0|)$ where $O_2 \subset \C$ is the open set from Lemma \ref{L:sin};
    \item $\| e^{\frac{\beta^2}{2} \sum_1^{n_0} h_n^2} - K_{n_0} \|_\infty$ is as small as desired, exploiting that $\sum_1^{n_0} h_n(x)^2$ is asymptotically independent of $x$ (the field $\Gamma$ is rotationally invariant);
    \item the chaos $:e^{i \beta (\sum_{n \geq n_0+1} A_n h_n)}:$ coming from the tail field is close to its expectation in $H^{-1/2-\eps}(\S^1)$-norm, with positive probability.
\end{itemize}
We can then conclude as before. More precisely, thanks to the first two properties, $\varphi_0$ is a small perturbation of the map $K_{n_0} F$ from Lemma \ref{L:sin}. Thus, there exist $(a_1, a_2) \in \R^2$ and a neighbourhood $B$ of $(a_1,a_2)$ such that $\varphi_0 : B \to \varphi(B)$ is a diffeomorphism and $\varphi_0(B)$ is a neighbourhood of $z_0$.
Using the third property, we can then conclude that this property is stable in the following sense.
Let $E \in \sigma(A_n, n \geq 3)$ be the event that $A_n$, $n =3, \dots, n_0$, stays close to 0 and that the $H^{-1/2-\eps}(\S^1)$-norm of $:e^{i \beta (\sum_{n \geq n_0+1} A_n h_n)}:$ is close to its expectation. The event $E$ occurs with positive probability and, on this event, the map
\[
\Phi: (s_1,s_2) \in \R^2 \longmapsto \psi(s_1, s_2, A_3, A_4, \dots)
\]
is also a diffeomorphism $\tilde B \to \Phi(\tilde B)$ where $\Phi(\tilde B)$ contains a neighbourhood of $z_0$. Altogether, this allows us to conclude as before that:
\begin{equation}
    \lim_{r \to 0^+} r^{-2} \P \Big( \Big| \int_{\S^1} : e^{i \beta \Gamma}: - z_0 \Big| < r \Big) > 0.
\end{equation}

\end{document}